\documentclass{amsart}

\usepackage{amsmath,amssymb,graphicx,mathrsfs}

\usepackage{color}

\numberwithin{equation}{section}

\usepackage{setspace}

\newcommand{\Gal}{\operatorname{Gal}}

\newcommand{\Aut}{\operatorname{Aut}}

\newcommand{\p}{\mathfrak{p}}

\newcommand{\OK}{\mathcal{O}_{K}}

\newcommand{\Z}{\mathbf{Z}}
\newcommand{\F}{\mathbf{F}}

\newcommand{\Q}{\mathbf{Q}}
\newcommand{\PP}{\mathbf{P}}
\newcommand{\A}{\mathbf{A}}

\newcommand{\PSL}{\operatorname{PSL}}
\newcommand{\PGL}{\operatorname{PGL}}

\newcommand{\GL}{\operatorname{GL}}
\newcommand{\SL}{\operatorname{SL}}

\newcommand{\im}{\operatorname{im}}

\newcommand{\Tr}{\operatorname{Tr}}
\newcommand{\Norm}{\operatorname{N}}

\newtheorem{thm}{Theorem}[section]

\newtheorem{lem}[thm]{Lemma}
\newtheorem{defn}[thm]{Definition}
\newtheorem{cor}[thm]{Corollary}
\newtheorem{prop}[thm]{Proposition}

\theoremstyle{definition}
\newtheorem*{rmk}{Remark}

\newcommand{\leg}[2]{\left(\frac{#1}{#2}\right)}
\newcommand{\into}{\hookrightarrow}
\newcommand{\rhobar}{\overline{\rho}}
\newcommand{\Cns}{C_{ns}}
\newcommand{\Csp}{C_{sp}}
\newcommand{\Nns}{N_{ns}}
\newcommand{\Nsp}{N_{sp}}
\newcommand{\G}{\operatorname{G}}
\newcommand{\Kbar}{\overline{K}}

\usepackage{multirow}

\begin{document}

\title{Local-Global principles for certain images of Galois representations}
\author{Anastassia Etropolski}

\begin{abstract}
Let $K$ be a number field and let $E/K$ be an elliptic curve whose mod $\ell$ Galois representation locally has image contained in a group $G$, up to conjugacy. We classify the possible images for the global Galois representation in the case where $G$ is a Cartan subgroup or the normalizer of a Cartan subgroup. When $K = \Q$, we deduce a counterexample to the local-global principle in the case where $G$ is the normalizer of a split Cartan and $\ell = 13$. In particular, there are at least three elliptic curves (up to twist) over $\Q$ whose mod $13$ image of Galois is locally contained in the normalizer of a split Cartan, but whose global image is not. 
\end{abstract}
\maketitle

\section{Introduction}
Let $E$ be an elliptic curve defined over a number field $K$. For a prime $\ell$, the points of order $\ell$ defined over $\Kbar$ form a rank two $\Z/\ell$-module, and it is natural to ask whether $E$ has a point of exact order $\ell$ defined over $K$. If this is the case, then the reduction of $E$ modulo a prime $\p$ coprime to $\ell$, denoted $\widetilde{E}_{\p}$,  will automatically have a point of order $\ell$. The converse to this is the following local-global problem: If $\widetilde{E}_{\p}$ has a point of order $\ell$ for almost all $\p$, does $E$ have a point of order $\ell$ defined over $K$? Katz studied this problem in \cite{katz1981} not only for elliptic curves but for higher dimension abelian varieties as well. In the case of elliptic curves, he showed that this is not true in general, but it is true that $E$ must be isogenous (over $K$) to an elliptic curve with a $K$-point of order $\ell$. 

One may rephrase this question in terms of the image of the mod $\ell$ Galois representation attached to $E$, denoted $\rhobar_{E,\ell}$. It turns out that $E$ having an $\ell$ torsion point over $K$ is equivalent to the image of the mod $\ell$ Galois representation landing in a certain subgroup of $\GL_2(\F_{\ell})$ (up to a choice of basis). The condition that $\widetilde{E}_{\p}$ have an $\ell$ torsion point is equivalent to the restriction of this representation to $\Gal(\overline{K_{\p}}/K_{\p})$ also landing in this type of subgroup. This allows one to rephrase the local-global problem entirely in the language of images of Galois representations. 

One natural subgroup of $\GL_2(\F_{\ell})$ is the group of upper triangular matrices. Similar to the story for torsion points, the image of the mod $\ell$ Galois representation lands in a group conjugate to the group of upper triangular matrices if and only if $E$ admits an isogeny of degree $\ell$, i.e.~ there exists an elliptic curve $E'/K$ and a degree $\ell$ isogeny $E \to E'$ defined over $K$. Sutherland studied the local-global problem for degree $\ell$ isogenies in \cite{sutherland2012} and showed that if $E/\Q$ admits a degree $\ell$ isogeny modulo $p$ for almost all $p$, then $E$ admits a degree $\ell$ isogeny over $\Q$, with exactly one exception: if $E$ has $j$-invariant $2268945/128$ and $\ell = 7$. This surprising counterexample comes from the fact that a certain modular curve has exactly one noncuspidal, non CM rational point. This type of argument is laid out in section \ref{sec:modcurves}. Sutherland also proved results in this direction for more general number fields, and his results have been generalized by others (see \cite{cremona-banwait2013}, \cite{anni2014}, \cite{vogt}).

This paper will generalize in a different direction by expanding the problem to other subgroups of $\GL_2(\F_{\ell})$, namely Cartan subgroups and their normalizers. A well known classification of the subgroups of $\GL_2(\F_{\ell})$ tells us that a subgroup of order prime to $\ell$ is either contained in a Cartan subgroup, the normalizer of a Cartan subgroup, or is one of the ``exceptional subgroups,'' which are small and well understood. In Section \ref{sec:local-global}, we determine when a local-global principle is allowed to hold via group theoretic considerations in the case of a general number field. In Section \ref{sec:modcurves} we take advantage of a calculation done by Banwait and Cremona in \cite{cremona-banwait2013} of some rational and quadratic points on the modular curve $X_{S_4}(13)$ to confirm some counterexamples to the local-global principle in the case where the image of the mod 13 Galois representation is locally contained in the normalizer of a split Cartan. This allows us to deduce a fairly complete theorem in the case of $K = \Q$. Full knowledge of $X_{S_4}(13)(\Q)$ is required to fully understand the failure of the local-global principle when $\ell = 13$.

\begin{thm}
Let $E/\Q$ be an elliptic curve and let $\ell$ be a prime. Let $G \subseteq \GL_2(\F_{\ell})$ be a fixed nonexceptional subgroup of order prime to $\ell$. If $\ell \ne 7, 13$ and the image of $\rhobar_{E,\ell}$ restricted to $\Gal(\overline{\Q_p}/\Q_p)$ is contained in $G$ up to conjugacy for almost all primes $p$, then $\im(\rhobar_{E,\ell})$ is contained in $G$ up to conjugacy. 

If $\ell = 7$, the only exception occurs when $G$ is a split Cartan, and it only occurs if $j(E) = 2268945/128$.

If $\ell = 13$, the only exception occurs when $G$ is the normalizer of a split Cartan, and there are at least 3 $j$-invariants classifying the isomorphism class containing $E$:
\begin{align*}
j(E) & = \frac{2^4 \cdot 5 \cdot 13^4 \cdot 17^3}{3^{13}},\\
j(E) & = -\frac{2^{12} \cdot 5^3 \cdot 11 \cdot 13^4}{3^{13}}, \text{ and} \\
j(E) & = \frac{2^{18} \cdot 3^3 \cdot 13^4 \cdot 127^3 \cdot 139^3 \cdot 157^3 \cdot 283^3 \cdot 929}{5^{13} \cdot 61^{31}}.
\end{align*}
\end{thm}

\section{Preliminaries}

\subsection{Galois representations and the Chebotarev density theorem}\label{sec:cdt}
Fix a prime number $\ell$, a number field $K$, and an algebraic closure $\Kbar$ of $K$. Then $\G_K := \Gal(\overline{K}/K)$ acts on the $\ell$-torsion points of $E(\Kbar)$, denoted $E[\ell]$, giving rise to the mod $\ell$ Galois representation
$$\rhobar_{E,\ell} \colon \G_{K} \to \Aut(E[\ell]) \simeq \GL_2(\F_{\ell}).$$

The Weil pairing on $E$ tells us that the composition of $\rhobar_{E,\ell}$ with the determinant map $\GL_2(\F_{\ell}) \to \F_{\ell}^{\times}$ is exactly the mod $\ell$ cyclotomic character. Therefore, if $E$ is defined over $K$ such that $K \cap \Q(\mu_{\ell}) = \Q$, the determinant map is necessarily surjective for all $\ell$. Furthermore, the image of the determinant map is contained in $\left(\F_{\ell}^{\times}\right)^2$ if and only if $K$ contains the unique quadratic subextension of $\Q(\mu_{\ell})/\Q$, namely $\Q(\sqrt{\ell^*}) := \Q\left(\sqrt{\leg{-1}{\ell} \ell}\right)$.

If $S$ is a finite set of primes of $K$ containing the primes of bad reduction and the primes above $\ell$, then $\rhobar_{E,\ell}$ is unramified outside of $S$. Therefore for $\p \notin S$, the restriction of $\rhobar_{E,\ell}$ to $\G_{K_{\p}}$ factors through $\widehat{\Z}$. Let $\varphi_{\p}$ denote a lift of the Frobenius automorphism of the residue field of $K_{\p}$. Then $\widehat{\Z}$ is topologically generated by $\varphi_{\p}$, and we denote the image of $\varphi_{\p}$ in $\GL_2(\F_{\ell})$ as a conjugacy class $\varphi_{\p,\ell}$. 

Letting $G \subseteq \GL_2(\F_{\ell})$, the Chebotarev density theorem implies that the set of $\p$ for which $\varphi_{\p,\ell}$ is contained in $G$ has positive density.  This allows us to set up the local-global problem as a purely group theoretic one. First we have the following definition.

\begin{defn}
We say that $E$ \emph{satisfies the local condition for $H$} if the image of the restriction of $\rhobar_{E,\ell}$ to $\G_{K_{\p}}$ is contained in a subgroup conjugate to $H$ for a set of primes $\p$ of density one. 
\end{defn}

Let $G = \rhobar_{E,\ell}(\G_K)$, and assume that $E$ satisfies the local condition for $H \subseteq \GL_2(\F_{\ell})$. Then, by the Chebotarev Density Theorem, for every $g \in G$, the conjugacy class of $g$ is equal to $\varphi_{\p,\ell}$ for a set of primes of positive density. Thus we may choose $\p$ such that $E$ satisfies the local condition for $H$ and the conjugacy class of $g$ coincides with $\varphi_{\p,\ell}$, i.e.~ $g$ is contained in a subgroup conjugate to $H$. The global condition is that $G$ be contained in a subgroup conjugate to $H$, so we can rephrase the problem as follows:
\begin{center}``If every $g \in G$ is contained in a group conjugate to $H$, is $G$ conjugate to a subgroup of $H$?''\end{center}

If the answer is ``Yes'', then we say that \emph{$E$ satisfies the local-global principle for $H$}. 

\begin{rmk}
If $\ell = 2$, then $E$ necessarily satisfies the local-global principle for every $H$ we will consider. This is because inside $\GL_2(\F_{\ell})$ the conjugacy class of both the split and nonsplit Cartan subgroup contains only a single element. For simplicity, we will assume in the proofs that $\ell > 2$.
\end{rmk}

\subsection{Subgroups of $\GL_2(\F_{\ell})$}\label{sec:subgroups}
In this section we will define some classical subgroups of $\GL_2(\F_{\ell})$.

A Borel subgroup is any subgroup conjugate to the subgroup of upper triangular matrices in $\GL_2(\F_{\ell})$, and therefore has order $\ell (\ell - 1)^2$. A Cartan subgroup comes in two varieties: split and nonsplit. A split Cartan is a group conjugate to the group of diagonal matrices, which we denote by $\Csp$, and is isomorphic to $(\F_{\ell}^{\times})^2$. A nonsplit Cartan is a group conjugate to
$$\Cns := \left\{  \begin{pmatrix} a & \delta b \\ b & a \end{pmatrix} \right\}\subseteq \GL_2(\F_{\ell})$$
for some $\delta$ with $\leg{\delta}{\ell} = -1$ and is isomorphic to $\F_{\ell^2}^{\times}$.\smallskip

Any Cartan subgroup has index 2 in its normalizer, and we have the following explicit constructions of their normalizers. Define the following subgroups of $\GL_2(\F_{\ell})$ by
$$
N_{sp} := \left\{ \begin{pmatrix} a & 0 \\ 0 & b \end{pmatrix},  \begin{pmatrix} 0 & c \\ d & 0 \end{pmatrix} \right\} 
$$
$$
N_{ns} := \left\{  \begin{pmatrix} a & \delta b \\ b & a \end{pmatrix},  \begin{pmatrix} a & -\delta b \\ b & -a\end{pmatrix} \right\},\; \text{where $\delta$ is any fixed quadratic nonresidue mod $\ell$}.
$$
Then the normalizer of a split Cartan will be conjugate to $\Nsp$ and the normalizer of a nonsplit Cartan will be conjugate to $\Nns$.

Alternatively, $\Nsp$ can be defined by adjoining $\left( \begin{smallmatrix}  0 & 1 \\ 1 & 0 \end{smallmatrix}\right)$ to $\Csp$, and $\Nns$ is the group obtained by adjoining $\left( \begin{smallmatrix}  0 & 1 \\ -1 & 0 \end{smallmatrix}\right)$ to $\Cns$. It is worth noting that a Borel subgroup is maximal, as is the normalizer of any Cartan subgroup.

These constructions will be useful for simplifying computations when we are allowed to fix a basis. More generally, we can define these subgroups by considering the action of $\GL_2(\F_{\ell})$ on $\PP^1(\F_{\ell})$ and $\PP^1(\F_{\ell^2})$, which we will define on the left as follows:
$$\begin{pmatrix} a & b \\ c & d \end{pmatrix}[x : y] := [ax + by : cx + dy].$$
If we restrict this action to the quotient $\PGL_2(\F_{\ell})$, then the action is faithful.

Let $g \in \GL_2(\F_{\ell})$. Then $g$ belongs to a Borel subgroup if it fixes a line in $\PP^1(\F_{\ell})$, it belongs to a split Cartan subgroup (resp. its normalizer) if it fixes (resp. fixes or swaps) two lines in $\PP^1(\F_{\ell})$, and it belongs to a nonsplit Cartan (resp. its normalizer) if it fixes (resp. fixes or swaps) two conjugate lines in $\PP^1(\F_{\ell^2}) \setminus \PP^1(\F_{\ell})$ for any fixed quadratic extension $\F_{\ell^2}/\F_{\ell}$.

We will restrict our attention to the Cartan subgroups and their normalizers. By definition, $g$ belongs to a split Cartan subgroup if and only if it is diagonalizable over $\F_{\ell}$, and two elements belong to the same split Cartan if and only if they are diagonalizable with respect to the same basis. 

To understand the nonsplit Cartan we need to fix a quadratic extension $\F_{\ell^2}/\F_{\ell}$. Up to scaling, we may fix a basis of the form $\{1, \alpha\}$ for this extension. Then $g$ is in the nonsplit Cartan corresponding to this basis if it fixes the line $[1 \colon \alpha]$ (it will necessarily also fix its conjugate since $g$ is defined over $\F_{\ell}$). If we write $g = \left( \begin{smallmatrix} a & b \\ c & d \end{smallmatrix}\right)$, this occurs if and only if $b \neq 0 $, $\Tr(\alpha) = (d-a)/b,$ and $\Norm(\alpha) = -c/b$. Now we may determine whether two elements belong to the same nonsplit Cartan by determining whether there exits an $\alpha \in \F_{\ell^2} \setminus \F_{\ell}$ that satisfies the corresponding norm and trace conditions for each element.

For $g$ to be in the normalizer of the nonsplit Cartan, but not necessarily in the Cartan itself, we need $g$ to swap $[ 1 \colon \alpha]$ and its conjugate. This occurs if and only if $a \Tr(\alpha) + b \Norm(\alpha) - c = 0$ (note that $g$ must have trace 0 for this to occur, so $d = -a$).

Using these descriptions, the following observations can be easily deduced. We combine them into one proposition for convenience.

\begin{prop}\label{prop:cartan facts}
Let $G \subseteq \GL_2(\F_{\ell})$. Denote by $H$ the image of $G$ in $\PGL_2(\F_{\ell})$, and for any $g \in \GL_2(\F_{\ell})$, $h$ will denote a representative for its image in $\PGL_2(\F_{\ell})$. Then the following are true:
\begin{enumerate}
\item If $g$ is diagonalizable, then $g$ is in a split Cartan.
\item If $g$ has irreducible characteristic polynomial, then $g$ is in a nonsplit Cartan.
\item Let $C$ be any Cartan subgroup and let $N$ denote its normalizer. Then for any $g \in N \setminus C$, $h$ has order two.
\item If $g$ is in the normalizer of a Cartan subgroup, then $g$ is diagonalizable over $\F_{\ell}$ if and only if its characteristic polynomial is reducible. 
\item If $g$ has trace 0, then $g$ is in the normalizer of a nonsplit Cartan. 
\end{enumerate}
\end{prop}

Now we can state the following classification of subgroups of $\GL_2(\F_{\ell})$.

\begin{prop}[{\cite[Lemma 2]{swinn1972}}]
\label{prop:maxsubgroups}
Let $G$ be a subgroup of $\GL_2(\F_{\ell})$. If $\ell \mid |G|$, then either $G$ is contained in a Borel or $G$ contains $\SL_2(\F_{\ell})$. If $\ell \nmid |G|$, let $H$ be the image of $G$ in $\PGL_2(\F_{\ell})$. Then
\begin{enumerate}

\item $H$ is cyclic and $G$ is contained in a Cartan subgroup, or
\item $H$ is dihedral and $G$ is contained in the normalizer of a Cartan subgroup but not in the Cartan subgroup itself, or
\item $H$ is isomorphic to $A_4, S_4, or A_5$.
\end{enumerate}
\end{prop}

\begin{rmk} The copies of $A_4$ and $A_5$ which appear in $\PGL_2(\F_{\ell})$ will actually be contained in $\PSL_2(\F_{\ell})$. Therefore, any lift to $\GL_2(\F_{\ell})$ of these groups will have determinant contained in $\left(\F_{\ell}^{\times}\right)^2$. If $\ell \equiv \pm 1\pmod{8}$, then the copy of $S_4$ will also be in $\PSL_2(\F_{\ell})$. If $\ell \equiv \pm 3\pmod{8}$, however, then $\PSL_2(\F_{\ell})$ does not contain a subgroup isomorphic to $S_4$, and so the determinant map will be nontrivial. In particular, there is a lift of $S_4$ in $\GL_2(\F_{\ell})$ with surjective determinant. \end{rmk}

Much of this paper will deal with the final category of subgroups in this list, which are called the exceptional subgroups. It is important to note that, for $\ell > 5$, subgroups $H$ which are isomorphic to one of these symmetric or alternating groups are all conjugate in $\PGL_2(\F_{\ell})$ (see \cite[Thm. 4.2]{beauville2010}). 

The classification above is really a consequence of the fact that the action of $\PGL_2(\F_{\ell})$ on $\PP^1(\F_{\ell})$ is very restricted, as we see in the following proposition.

\begin{prop}[{\cite[Prop.~2]{sutherland2012}}]\label{prop:drew}
Let $g \in \GL_2(\F_{\ell})$ have image $h \in \PGL_2(\F_{\ell})$ with order $r$, let $k$ be the number of lines in $\PP^1(\F_{\ell})$ fixed by $h$, and let $s$ be the number of $h$-orbits under this action. Then $k$ is $0, 1, 2,$ or $\ell + 1$, and the $s-k$ nontrivial $h$-orbits have size $r$. When $\ell > 2$ we also have $\sigma(h) = (-1)^s,$ where $\sigma(h)$ is the sign of $h$ as a permutation of $\PP^1(\F_{\ell})$.
\end{prop}

This proposition will be used extensively throughout, so we lay out the general argument here: Suppose that $h$ is as above and we know that $h$ swaps a pair of lines. Then, apart from the elements that $h$ fixes, we know that $h$ only swaps pairs of lines. Moreover, $h$ must have order $2$. The sign of $h$ will still depend on $\ell$ and $k$: For example, if $h$ fixes 2 lines and swaps the remaining $\ell - 1$, then $\sigma(h) = (-1)^s$, where $s = \frac{\ell - 1}{2} + 2$. Thus $\sigma(h) = 1$ if and only if $\ell \equiv 1\pmod{4}$. What will generally occur is a sort of converse of this. When $H$ is one of $A_4,$ $A_5$, or $S_4$, $\sigma(h)$ will (in most cases) be determined by the order of $h$, and this will give a congruence condition that $\ell$ must satisfy. 

Suppose that $H \simeq A_4$ or $A_5$. Since $A_4$ and $A_5$ have only the trivial homomorphism to $\{\pm 1\}$, $\sigma$ must be trivial. This means that elements of the same order must fix the same number of lines in $\PP^1(\F_{\ell})$. On the other hand, $S_4$ has two maps to $\{\pm 1\}$, namely the trivial one and the usual sign map on $S_4$. If $\ell \equiv \pm 1 \pmod{8}$, then every element of $H \simeq S_4$ will have determinant 1, and so $\sigma$ will be the trivial map (since $\sigma(h) = 1$ if and only if $h \in \PSL_2(\F_{\ell})$). When $\ell \equiv \pm 3 \pmod{8}$, the determinant map will be nontrivial, so $\sigma$ will be the usual sign map on $S_4$, i.e.~$\sigma(g) = 1$ if $g$ has order 3, $\sigma(g) = -1$ if $g$ has order 4, and $\sigma(g)$ can be either $\pm 1$ when the order of $g$ is 2. Therefore elements of order $r$ when $r= 3$ or $4$ must all fix the same number of lines, but elements of order 2 are forced to act differently.

\subsection{Image of Inertia}\label{sec:inertia}

In \cite{Serre1972}, Serre explicitly worked out the possible images of inertia under the mod $\ell$ Galois representation. Using this knowledge, we are able to better understand when we can rule out the exceptional subgroups. The results of this section are immediate from the work of Serre, but do not seem to be stated elsewhere in this generality, so we state them here. The following proposition, as reformulated by Mazur, captures the results that we need.

\begin{prop}[{\cite[\textsection 2, Remark 2]{mazur1977}}]\label{prop:inertia}
Let $K$ be a finite extension of $\Q_{\ell}$ of ramification index $e$. Let $E$ be an elliptic curve over $K$ with semistable N\'{e}ron model over the ring of integers $\OK$. Let $r \colon \Gal(\Kbar/K) \to \PGL_2(\F_{\ell})$ denote the projective representation associated to the action of Galois on the $\ell$-division points of $E$. Then, if $2e < \ell - 1$, the image of the inertia subgroup under $r$ contains an element of order $\ge (\ell - 1)/e$.
\end{prop}

We may now state explicit bounds for when the image of $\rhobar_{E,\ell}$ is one of the exceptional subgroups. 

\begin{prop}\label{prop:bounds}
Let $K$ be a number field of degree $d$ and let $E/K$ be an elliptic curve. Fix a prime $\ell > 3$, let $G$ be the image of $\rhobar_{E,\ell}$ in $\GL_2(\F_{\ell})$, and let $H$ be its image in $\PGL_2(\F_{\ell})$. Then we have the following:
\begin{enumerate}
\item If $H \simeq A_4$, then $\ell \le 9d + 1$.
\item If $H \simeq S_4$, then $\ell \le 12d + 1$.
\item If $H \simeq A_5$, then $\ell \le 15d + 1$.
\end{enumerate}
\end{prop}
\begin{proof}
 Let $K_{\lambda}$ be the completion of $K$ at a prime $\lambda$ above $\ell$, and let $M$ be the smallest extension of $K_{\lambda}$ over which $E_{\lambda} := E \otimes K_{\lambda}$, or a quadratic twist of $E_{\lambda}$, obtains semistable reduction. Since quadratic twisting preserves the projective image of Galois, we may replace $E$ by a quadratic twist if necessary. Let $v$ be the valuation corresponding to the unique place above $\lambda$ in $M$. For $\ell > 3$, we know that $[M : K_{\lambda}] \le 3$ (see for example \cite{anni2014},  Section 4.2), so $e = v(\ell) \le 3d$.  Then the base extension of $E_{\lambda}$ to $M$ has semistable reduction at $\ell$.

Fix $H$ to be either $A_4, S_4,$ or $A_5$, and define $h_H  := \max\{ |h| \colon h \in H\}$. Then the previous proposition tells us that if $2e < \ell - 1$ and $H$ is the image of $r$, we must have that
$$\frac{\ell - 1}{e} \le h_H,$$
and so we can conclude that $\ell \le 3d h_H + 1$. Plugging in $h_H = 3,4,5$ respectively gives the bounds above.

Note that if $2e \ge \ell - 1$, then we get the bound $\ell \le 2e + 1 \le 6d + 1$ which is automatically included in the bounds we produced.
\end{proof}

\begin{rmk}
These bounds are necessarily general and can be improved in certain cases. In particular, equality assumes that $e$ is as large as possible and that $(\ell - 1)/e$ is the order of the largest element of $H$.
\end{rmk}

\begin{cor}
If $E/\Q$ is an elliptic curve and the image of $\rhobar_{E,\ell}$ is exceptional, then $\ell \le 13$. 
\end{cor}
\begin{proof}
As remarked earlier, the only exceptional subgroup which can occurs as the image of $\rhobar_{E,\ell}$ is the one corresponding to $S_4$. The corollary then follows directly from the proposition.
\end{proof}

\section{Local-Global principle for subgroups of $\GL_2(\F_{\ell})$}\label{sec:local-global}

\subsection{Split and nonsplit Cartan} \label{Csp}

Let $G \subseteq \GL_2(\F_{\ell})$ and suppose that every $g \in G$ is diagonalizable. This is equivalent to saying that every $g$ is contained in some split Cartan group. The local-global question is whether $G$ itself is contained in a split Cartan, i.e.~whether the $g$ are simultaneously diagonalizable. 

The answer to this question follows from \cite{sutherland2012} and \cite{cremona-banwait2013}. In their case, every $g$ is contained in a Borel (i.e.~fixes one line). Since the split Cartan is contained in the Borel, we may apply their results to this case. 

\begin{cor}\label{corr:Csp}
Let $K$ be a number field and let $E/K$ be an elliptic curve. Let $G = \im(\rhobar_{E,\ell})$ and let $H$ be its image in $\PGL_2(\F_{\ell})$. Suppose that $E$ satisfies the local condition for the split Cartan. Then either $G$ is contained in a split Cartan or one of the following is true:
\begin{enumerate}
\item $G$ is contained in the normalizer of a split Cartan but not the Cartan itself.
\item $H \simeq A_4$, with $\ell \equiv 1 \pmod{12}$.
\item $H \simeq S_4$, with $\ell \equiv 1 \pmod{24}$
\item $H \simeq A_5$, with $\ell \equiv 1 \pmod{60}$.
\end{enumerate}
If $K \cap \Q(\mu_{\ell}) = \Q$, then $(2) - (4)$ cannot occur, and $(1)$ can occur only if $\ell \equiv 3\pmod{4}$. If $K \cap \Q(\mu_{\ell}) 
\ne \Q$ and one of $(1)$ - $(4)$ is true, we must have that $\ell \equiv 1\pmod{4}$ and $K$ contains $\Q(\sqrt{\ell})$.
\end{cor}

\begin{proof}
This result follows immediately from \cite[Lemma 1 and Theorem 1]{sutherland2012} and \cite[Proposition 1.10]{cremona-banwait2013}. The exceptional subgroups which arise in \cite[Proposition 1.10]{cremona-banwait2013} are the same as the ones above because the congruence conditions actually imply that each element is contained in a split Cartan, not just a Borel.
\end{proof}

\begin{cor}
If $\ell \ne 7$, then $E/\Q$ satisfies the local-global principle for the split Cartan. If $\ell = 7$, then $E$ satisfies the local-global principle for the split Cartan if and only if $j(E) \ne 2268945/128$.
\end{cor}
\begin{proof}
This is a direct consequence of the previous corollary in conjunction with Section 3 and Theorem 2 in \cite{sutherland2012}. An elliptic curve over $\Q$ with $j(E) \ne 2268945/128$ actually admits two $7$-isogenies modulo every prime of good reduction, so it does indeed satisfy our stronger local condition.
\end{proof}

The following theorem explains what happens in the \emph{nonsplit} Cartan case. We use the description in Section \ref{sec:subgroups} to analyze the local condition. First, we need to restrict our attention to number fields $K$ which have no real embeddings. This is because if $K$ is totally real, then complex conjugation acts with eigenvalues $\pm 1$, and no element of a nonsplit Cartan has those eigenvalues. So the only way that $E/K$ can satisfy the local condition for the nonsplit Cartan is if complex conjugation acts trivially on $K$. 

\begin{thm}\label{thm:Cns}
Let $K$ be an imaginary number field and let $E/K$ be an elliptic curve. Let $G = \im(\rhobar_{E,\ell})$ and let $H$ be its image in $\PGL_2(\F_{\ell})$. Suppose that $E$ satisfies the local condition for the nonsplit Cartan. Then either $G$ is contained in a nonsplit Cartan or one of the following is true:
\begin{enumerate}
\item $G$ is contained in the normalizer of a nonsplit Cartan but not the Cartan itself.
\item $H \simeq A_4$ and $\ell \equiv -1 \pmod{12}$.
\item $H \simeq S_4$ and $\ell \equiv -1 \pmod{24}$.
\item $H \simeq A_5$ and $\ell \equiv -1 \pmod{60}$.
\end{enumerate}
If $K \cap \Q(\mu_{\ell}) = \Q$, then $(2) - (4)$ cannot occur, and $(1)$ can occur only if $\ell \equiv 1\pmod{4}$. If $K \cap \Q(\mu_{\ell}) 
\ne \Q$ and one of $(1) - (4)$ is true, we must have that $K$ contains $\Q(\sqrt{\ell^*})$.
\end{thm}
\begin{proof}
Suppose that every $g \in G$ is contained in some nonsplit Cartan. Equivalently, $g$ is either scalar, or it has irreducible characteristic polynomial. As every $g$ has order dividing $\ell^2 - 1$, we know that $\ell \nmid |G|$, and so $G$ is either contained in a Cartan, in the normalizer of a Cartan, or is one of the exceptional subgroups.

Suppose that $G$ is contained in a Cartan subgroup. Then $H$ is cyclic, so $G$ must be contained in a nonsplit Cartan since the action of the group is determined by the action of a generator.

Now suppose that $G$ is contained in the normalizer of a Cartan subgroup. Recall that a matrix $g = \left(\begin{smallmatrix} a & b \\ c & d\end{smallmatrix} \right)$ is in the nonsplit Cartan corresponding to the basis $\{1, \alpha\}$ if and only if $b \ne 0$, $\Tr(\alpha) = (d-a)/b$, and $\Norm(\alpha) = -c/b$. If $G$ is contained in the normalizer of a split Cartan, then by fixing a basis we may assume that each matrix in $G$ is either diagonal or antidiagonal. The only diagonal matrices which are contained in a nonsplit Cartan are those which are scalar, so we may assume that $G$ contains antidiagonal matrices. Suppose that $\left(\begin{smallmatrix} 0 & a \\ b & 0\end{smallmatrix} \right)$ and $\left(\begin{smallmatrix} 0 & c \\ d & 0\end{smallmatrix} \right)$ are two elements of $G$. Then their product is $\left(\begin{smallmatrix} ad & 0 \\ 0 & bc\end{smallmatrix} \right)$, which must be scalar, so $ad = bc$. Under this condition, it is easy to check that both matrices are in a nonsplit Cartan corresponding to the same basis: in particular any basis $\{1, \alpha\}$ with $\Tr(\alpha) = 0$, $\Norm(\alpha) = -b/a =  -d/c$. 

Now we examine the possibility that $G$ is contained in the normalizer of a nonsplit Cartan. We will show that if this is the case, and $G$ is not contained in a nonsplit Cartan, then such a $G$ will have surjective determinant only if $\ell \equiv 1\pmod{4}$, and if $\ell \equiv 3\pmod{4}$, then the determinant of each $g \in G$ will be a square.

Suppose that $G$ is contained in the normalizer of a nonsplit Cartan but not in the nonsplit Cartan itself. We will fix a basis so that $G \subset \Nns$, as defined in Section \ref{sec:subgroups}. Recall that $C_{ns}$ is cyclic and of index two in its normalizer, so $C = C_{ns} \cap G$ is cyclic. Then $G / C \into \Z/2\Z$, so $G$ is generated by at most two elements. If $G$ is cyclic, then we are in the previous case, so we may assume that $G$ has two generators, and in particular, we can choose one generator to be in $\Nns \setminus \Cns$ and the other to be in $\Cns$. Fix $\leg{\delta}{\ell} = -1$ and let $g = \left(\begin{smallmatrix} a & \delta b \\ b & a \end{smallmatrix}\right)$, $b \ne 0$, and let $h = \left(\begin{smallmatrix} x & -\delta y \\ y & -x \end{smallmatrix}\right)$, $y \ne 0$, so that $G = \langle g, h \rangle$. Since we are assuming that every element of $G$ is in some nonsplit Cartan, we need $h$ to have irreducible characteristic polynomial. This occurs precisely when $-\det(h)$ is not a square.

One may easily verify that $g$ belongs to a nonsplit Cartan corresponding to any basis $\{1, \alpha\}$ with $\Tr(\alpha) = 0$, $N(\alpha) = -1/\delta$, and $h$ belongs to a nonsplit Cartan corresponding to any basis $\{1, \beta\}$ with $\Tr(\beta) = 2x/\delta y$, $\Norm(\beta) = 1/\delta$. This shows that there is no nonsplit Cartan which contains both $g$ and $h$. Furthermore, in order for $G$ to have the property that every element is in some nonsplit Cartan, we need $gh$ to have irreducible characteristic polynomial, since it is an element of $\Nns \setminus \Cns$. Thus we need $-\det(gh) = -\det(g) \det(h)$ to not be a square, so $\det(g)$ is necessarily a square. In order for $G$ to have surjective determinant, $\det(h)$ must not be a square. This occurs if and only if $\ell \equiv 1\pmod{4}$. 

All that remains is to analyze the possible exceptional subgroups that could arise. To do this we will use Proposition \ref{prop:drew} to determine the congruence conditions that $\ell$ must satisfy for each exceptional subgroup. First, suppose that $H \simeq A_4$. The nontrivial elements of $A_4$ have order 2 and 3. Since are assuming that these elements do not fix any element of $\PP^1(\F_{\ell})$, we need $\frac{\ell + 1}{r}$ to be even for $r = 2, 3$, so $\ell \equiv -1 \pmod{12}$. 

Now suppose that $H \simeq A_5$. The nontrivial elements of $A_5$ have order 2, 3, and 5 and fix no elements of $\PP^1(\F_{\ell})$. Therefore we need $\frac{\ell + 1}{r}$ to be even for $r = 2, 3, 5$, so $\ell \equiv -1 \pmod{60}$.

Finally, suppose that $H \simeq S_4$. If $\ell \equiv \pm 3 \pmod{8}$, then the sign map on $H$ will be nontrivial, so it must be exactly the sign homomorphism on $S_4$. This means that $(\ell + 1)/4$ must be odd, but this forces the elements of order 2 to have the same sign, so this cannot occur.

If $\ell \equiv \pm 1 \pmod{8}$, then $H \subseteq \PSL_2(\F_{\ell})$. Therefore the sign of every element of $H$ will be 1, and we must have that $\frac{\ell + 1}{r}$ is even for $r = 2, 3, 4$. This gives the condition $\ell \equiv -1 \pmod{24}$.

\end{proof}

\subsection{Normalizer of a split Cartan}\label{sec:Nsp}

An element of the normalizer of a split Cartan has the property that it either fixes two lines in $\PP^1(\F_{\ell})$ or it swaps them. Recall from Proposition \ref{prop:cartan facts} that an element of the normalizer of a Cartan that is not in the Cartan itself will have order two in $\PGL_2(\F_{\ell})$. Applying Proposition \ref{prop:drew}, this means that such an element acts in one of the following two ways: If it is diagonalizable, then it fixes a pair of lines and swaps the remaining pairs, and if it is not diagonalizable, then it only swaps pairs of lines. This immediately shows that if $g$ is in the normalizer of a \emph{nonsplit} Cartan but not in the Cartan itself, then it is automatically in the normalizer of a split Cartan as well, so there are many elements of the normalizer of a nonsplit Cartan which will satisfy the local condition. 

The bulk of the following theorem is to understand when a group satisfying the local condition for the normalizer of a split Cartan is actually contained in the normalizer of a nonsplit Cartan instead. To avoid redundancy, we exclude the case where $E$ actually satisfies the local condition for the split Cartan.

\begin{thm}\label{thm:Nsp}
Let $K$ be a number field of degree $d$ and let $E/K$ be an elliptic curve. Let $G \subseteq \GL_2(\F_{\ell})$ denote the image of $\rhobar_{E,\ell}$ and let $H$ denote the image of $G$ in $\PGL_2(\F_{\ell})$. Suppose that $E$ satisfies the local condition for the normalizer of a split Cartan, but $E$ does not satisfy the local condition for the split Cartan. Then either $G$ is contained in the normalizer of a split Cartan or one of the following holds:
\begin{enumerate}
\item $G$ is contained in the normalizer of a nonsplit Cartan and $\ell \equiv 3 \pmod{4}$, with $\ell \le 6d  + 1$.
\item $H \simeq A_4$ and $\ell \equiv 7\pmod{12}$.
\item $H \simeq S_4$ and $\ell \equiv 13 \pmod{24}$.
\item $H \simeq A_5$ and $\ell \equiv 31 \pmod{60}$.  
\end{enumerate}
If $K \cap \Q(\mu_{\ell}) = \Q$, then only $(3)$ can occur, and if one of $(1)$, $(2)$, or $(4)$ holds, then $K$ contains $\Q(\sqrt{\ell^*})$.
\end{thm}
\begin{proof}

Suppose that every $g \in G$ is contained in some nonsplit Cartan and at least one element of $G$ is not contained in any nonsplit Cartan. Then $g$ has order dividing $2(\ell^2 - 1)$, so $\ell \nmid |G|$. Thus $G$ is either contained in a Cartan, the normalizer of a Cartan, or is one of the exceptional subgroups.

If $G$ is contained in a Cartan, then $G$ is cyclic. Since its generator is by assumption contained in the normalizer of a split Cartan, so is $G$ and we are done. 

Now suppose that $G$ is contained in the normalizer of a Cartan but not in the Cartan itself. We will show that if $G$ is contained in the normalizer of a nonsplit Cartan, then $\ell \equiv 3\pmod{4}$ with $\ell \le 6d + 1$, and the determinant map has image contained in $\left(\F_{\ell}^{\times}\right)^2$. Without loss of generality, assume that $G \subseteq \Nns$.

As we saw in the proof of Theorem \ref{thm:Cns}, $G$ is generated by at most two elements. We have already ruled out the cyclic case, so it remains to show the result when $G$ has two generators, and as before, we can choose one generator to be in $\Nns \setminus \Cns$ and the other to be in $\Cns$. Recall that every $g \in \Nns \setminus \Cns$ satisfies the local condition, and that if the element of $\Cns$ is diagonalizable, then it must be scalar, in which case the image in $\PGL_2(\F_{\ell})$ is cyclic and we are done. 

Let $G = \langle g, h \rangle$, where $g \in \Nns \setminus \Cns$, and $h \in \Cns$ is not scalar. Then we can write $g = \left(\begin{smallmatrix} a & -\delta b \\ b & -a \end{smallmatrix} \right)$ and $h = y \left(\begin{smallmatrix} 0 & \delta \\ 1 & 0 \end{smallmatrix}\right)$ for some $y \neq 0$. Notice that $g$ is diagonalizable if and only if $-\det{g}$ is a square modulo $\ell$.

Suppose first that $g$ is diagonalizable. Then it is easy to check that for any $L \in \PP^1(\F_{\ell})$, $g$ fixes $L$ if and only if it fixes $hL$. Therefore the pair of lines fixed by $g$ is swapped by $h$, so $g$ and $h$, belong to normalizer of the same split Cartan, so $G$ is contained in the normalizer of a split Cartan.

Now suppose that $g$ is not diagonalizable. If $a = 0$, then $g$ and $h$ both swap the two axes, and so $G$ is contained in the normalizer of a split Cartan. However, when $a = 0$, $\det g = \delta b^2$, which is not a square. Since we are assuming that $g$ is not diagonalizable, i.e.~ that $-\det g$ is not a square, we conclude that $\ell \equiv 1\pmod{4}$ in this case.

For $a \neq 0$ we proceed as follows. We have that $g$ and $h$ swap the same pair of lines if and only if $gL = hL$ for some $L$, or $h^{-1}gL = L$. If we let $L = [1 \colon m]$, then this will occur if and only if
$$a\delta m^2 - 2b\delta m + a = 0.$$
This polynomial has discriminant $4 \delta \det{g}$, which is a square if and only if $\ell \equiv 1\pmod{4}$, since $g$ is not diagonalizable. We conclude that $G$ is contained in the normalizer of a split Cartan if and only if $\ell \equiv 1\pmod{4}$. Now let us examine what happens when $\ell \equiv 3\pmod{4}$.

We know that, up to scaling the first generator, $G$ is conjugate to a group of the form
\begin{equation*}
\left\langle \begin{pmatrix} 0 & \delta \\ 1 & 0 \end{pmatrix} , \begin{pmatrix} a & -\delta b \\ b & -a \end{pmatrix}\right\rangle,
\end{equation*}
where neither $\delta$ nor $a^2 - \delta b^2 = -\det \left(\begin{smallmatrix} a & - \delta b \\ b & - a\end{smallmatrix}\right)$ is a square in $\F_{\ell}$. Thus every element of this group has determinant a square. Moreover, the image of this subgroup in $\PGL_2(\F_{\ell})$ is isomorphic to $\Z/2\Z \times \Z/2\Z$. Using our knowledge of the image of inertia as in the case of the exceptional subgroups, we conclude that in this case $\ell \le 6 d + 1$.

All that remains is to examine what happens when $G$ is one of the exceptional subgroups. First suppose that $H \simeq A_4$. Then $H$ has nontrivial elements of order 2 and 3. Since every element of $\Nsp \setminus \Csp$ has order 2 in $\PGL_2(\F_{\ell})$, the elements of order 3 must belong to a split Cartan, and therefore fix a pair of lines. Since the sign map is necessarily the trivial map, we need $\frac{\ell - 1}{3}$ to be even. Therefore $\ell \equiv 1\pmod{6}$. The elements of order 2 may entirely belong to either a split Cartan or its complement in the normalizer, but since we are assuming that $G$ contains at least one element which is not in any split Cartan, the elements of order 2 are forced to be nondiagonalizable. Therefore $\ell \equiv 7 \pmod{12}$.

Similarly, if $H \simeq A_5$, the elements of orders 3 and 5 must belong to the split Cartan, whereas the elements of order 2 may belong to either. Excluding the case where every element is contained in a split Cartan, this produces the condition $\ell \equiv 31 \pmod{60}$, where again, we must assume that the elements of order 2 are not diagonalizable. 

Finally, suppose that $H \simeq S_4$. Then $H$ has nontrivial elements of order 2, 3, and 4. Again, the elements of orders 3 and 4 are necessarily diagonalizable. If $\ell \equiv \pm 3\pmod{8}$, then the sign map on $H$ is nontrivial and must correspond to the sign homomorphism on $S_4$. This leads to the condition that $\ell \equiv 13\pmod{24}$. Notice that in this case, the elements of order 2 which are even are exactly the elements of order 2 that belong to split Cartan, and the elements of order 2 which are odd are exactly the ones that belong to its complement in the normalizer.

If $\ell \equiv \pm 1\pmod{8}$, then the sign map on $H$ is trivial, and this forces $\ell \equiv 1\pmod{24}$. In this case, however, every element of $H$ is diagonalizable, and so every element of $G$ is actually contained in a split Cartan. 
\end{proof}

\begin{cor} Let $E / \Q$ be an elliptic curve. Then $E$ satisfies the local-global principle for the normalizer of a split Cartan for all $\ell \ne 13$.
\end{cor}

In fact, in Section \ref{sec:modcurves} we will see that there are at least three counterexamples in the case of $\ell = 13$.

\subsection{Normalizer of a nonsplit Cartan}\label{sec:Nns}

If the image of $\rhobar_{E,\ell}$ is locally in the normalizer of a nonsplit Cartan, then every element of the image fixes or swaps a pair of conjugate lines of the form $[1 \colon \alpha] \in \PP^1(\F_{\ell^2})$.

We will begin by classifying the groups $G$ with the property that for all $g \in G$, $g$ is in the normalizer of a nonsplit Cartan, yet $G$ is contained in the normalizer of a \emph{split} Cartan. For simplicity, we will fix our basis so that $G \subseteq \Nsp$, i.e.~ $G$ consists only of diagonal and antidiagonal matrices. First we have the following lemma.

\begin{lem}\label{lem:NnsinNsp}
Let $G \subseteq \Nsp$ and suppose that every $g \in G$ is in the normalizer of some nonsplit Cartan. Then there is a nonsplit Cartan subgroup whose normalizer contains $G$.
\end{lem}
\begin{proof}
Up to scalar multiplication, there are only two types of matrices in $\Nsp$ that satisfy the assumption: $\left(\begin{smallmatrix} 1 & 0 \\ 0 & -1\end{smallmatrix}\right)$, which is in the normalizer of a nonsplit Cartan (but not in the Cartan itself) corresponding to any basis $\{1, \alpha\}$ where $\alpha$ has trace 0, and matrices of the form $\left(\begin{smallmatrix} 0 & 1 \\ a & 0\end{smallmatrix}\right)$, which are in the normalizer of a nonsplit Cartan corresponding to any basis $\{1, \alpha\}$ where $\Norm(\alpha) = a$, and in a nonsplit Cartan corresponding to any basis $\{1, \alpha\}$ where $\Tr(\alpha) = 0$ and $\Norm(\alpha) = -a$. 

If $G$ contains two matrices  $\left(\begin{smallmatrix} 0 & 1 \\ a & 0\end{smallmatrix}\right)$ and $\left(\begin{smallmatrix} 0 & 1 \\ b & 0\end{smallmatrix} \right)$, then the only way for their product to be of one of the two allowable forms is if $a = \pm b$, in which case there is a nonsplit Cartan whose normalizer contains both matrices, as well as the matrix $\left(\begin{smallmatrix} 1 & 0 \\ 0 & -1\end{smallmatrix}\right)$. 
\end{proof}

\begin{thm}\label{thm:Nns}
Let $K$ be a number field and let $E/K$ be an elliptic curve. Let $G \subseteq \GL_2(\F_{\ell})$ denote the image of $\rhobar_{E,\ell}$ and let $H$ denote the image of $G$ in $\PGL_2(\F_{\ell})$. Suppose that $E$ satisfies the local condition for the normalizer of a nonsplit Cartan but $E$ does not satisfy the local condition for the nonsplit Cartan. Then either $G$ is contained in the normalizer of a nonsplit Cartan or one of the following holds:
\begin{enumerate}
\item $H \simeq A_4$ and $\ell \equiv 5 \pmod{12}$.
\item $H \simeq S_4$ and $\ell \equiv 11 \pmod{24}$.
\item $H \simeq A_5$ and $\ell \equiv 29 \pmod{60}$.
\end{enumerate}
If $K \cap \Q(\mu_{\ell}) = \Q$, then $(1)$ and $(3)$ cannot occur, and if $(1)$ or $(3)$ holds, then $K$ contains $\Q(\sqrt{\ell^*})$.
\end{thm}
\begin{proof}
As in the proof of Theorem \ref{thm:Nsp}, we go through the possibilities for $H$ as enumerated in Proposition \ref{prop:maxsubgroups}. Suppose that $E$ satisfies the local condition for the normalizer of a nonsplit Cartan, i.e.~every $g \in G$ is in the normalizer of some nonsplit Cartan. Then the order of $G$ is prime to $\ell$ since the order of $\Nns$ is $2(\ell^2 - 1)$. Therefore $G$ is either contained in a Cartan subgroup, the normalizer of a Cartan subgroup, or is one of the exceptional subgroups. If $G$ is contained in a Cartan subgroup, then $H$ is cyclic, and so every element of $G$ is in fact in the normalizer of the same nonsplit Cartan subgroup and the global condition is satisfied. Furthermore, if $G$ is contained in the normalizer of a split Cartan, then Lemma \ref{lem:NnsinNsp} shows that $G$ is also contained in the normalizer of a nonsplit Cartan, thereby satisfying the global condition. 

Finally, we examine the exceptional subgroups. Observe that the only diagonalizable elements of the normalizer of a nonsplit Cartan are those which arise in the normalizer, rather than coming from the Cartan subgroup itself. These all have order 2 in $\PGL_2(\F_{\ell})$, so the elements of order 3, 4, and 5 which occur in the various exceptional subgroups must all be nondiagonalizable. The calculation proceeds as in the the proof of Theorem \ref{thm:Nsp}, and we throw out the cases where every element is actually contained in a nonsplit Cartan.
\end{proof}

\begin{cor}\label{cor:Nns over Q}
Let $E/\Q$ be an elliptic curve. Then $E$ satisfies the local-global principle for the normalizer of a nonsplit Cartan.
\end{cor}
\begin{proof}
By the previous theorem, the only way in which the local-global principle could fail to hold is if there exists an elliptic curve over $\Q$ whose mod 11 image of Galois is contained in the exceptional subgroup corresponding to $S_4$. It has been shown in \cite[II.4.4]{ligozat1976} that there is no such elliptic curve. 
\end{proof}

\section{Modular Curves and Specific Counterexamples}\label{sec:modcurves}

Given an integer $N$ and a subgroup $H \subseteq \GL_2(\Z/N\Z)$, there exists an algebraic curve $Y_{H}(N)$ and a map $j \colon Y_H(N) \to \A^1$ with the following property: If $P \in Y_H(N)(K)$, then there exists an elliptic curve $E/K$ with image of Galois contained in a subgroup conjugate to $H$ and $j(E) = j(P)$, where $j(E)$ is the $j$-invariant of $E$. Conversely, if $E/K$ is an elliptic curve whose image of Galois is contained in a subgroup conjugate to $H$, then there exists a point $P \in Y_H(N)(K)$ with $j(E) = j(P)$. There is a smooth compactification $X_{H}(N)$ of $Y_{H}(N)$, and we call $X_H(N)$ the \emph{modular curve of level $N$ associated to $H$}. The $K$-rational points of $X_H(N)$ coming from the compactification are called cusps and correspond to generalized elliptic curves in the sense of Deligne and Rapoport (see \cite{deligne-rapoport1973}). We are interested in studying the noncuspidal points in $X_H(N)(K)$, in particular in the case when $N = \ell$ is prime.

Theorem \ref{thm:Nsp} tells us that for $\ell = 13$, if there exists a counterexample over $\Q$ to the local-global principle for the normalizer of a split Cartan, then the image of the mod 13 Galois representation for that elliptic curve is contained in a subgroup $H_{S_4} \subseteq \GL_2(\F_{13})$ with image in $\PGL_2(\F_{13})$ isomorphic to $S_4$. To find out if such a curve exists, we consider the rational points of $X_{S_4}(13)$, which is the modular curve associated to $H_{S_4}$. It is worth nothing that a rational point on $X_{S_4}(\ell)$ does not necessarily satisfy the local condition for the normalizer of a split Cartan, but our congruence conditions on $\ell$ guarantee that this is the case, so it will hold for $\ell = 13$. Furthermore, a rational point on $X_{S_4}(\ell)$ could correspond to an elliptic curve whose image of Galois is strictly contained in $H_{S_4}$, so care must be taken to make sure that the image is the entire group.

This calculation was done by Banwait and Cremona in \cite[Corollary 1.5]{cremona-banwait2013}. They also found quadratic points on this curve, in which case the image of the Galois representation may be contained in $A_4$, by Theorem \ref{thm:Nsp}. The conclusions of their calculations are laid out in the following table.\bigskip

\begin{center}
\begin{tabular}{|c|c|l|l|}
\hline \multicolumn{4}{|c|}{} \\ [-3ex]
\multicolumn{4}{|c|}{\textbf{Confirmed Counterexamples to Local-Global for the Normalizer of the split Cartan}} \\
\hline & & &  \\ [-3ex]
$\ell$ & $K$ & $j$-invariant & $H$ \\
\hline & & &  \\ [-3ex]
$13$ & $\Q$ & $\displaystyle \frac{2^4 \cdot 5 \cdot 13^4 \cdot 17^3}{3^{13}}$ & $S_4$ \\
[1.5ex] \hline & & &  \\ [-3ex]
$13$ & $\Q$ & $\displaystyle -\frac{2^{12} \cdot 5^3 \cdot 11 \cdot 13^4}{3^{13}}$ & $S_4$ \\
[1.5ex] \hline & & &  \\ [-3ex]
$13$ & $\Q$ & $\displaystyle \frac{2^{18} \cdot 3^3 \cdot 13^4 \cdot 127^3 \cdot 139^3 \cdot 157^3 \cdot 283^3 \cdot 929}{5^{13} \cdot 61^{31}}$ & $S_4$ \\
[1.5ex] \hline & & &  \\ [-3ex]
$13$ & $\Q\left(\sqrt{13}\right)$ &  $\displaystyle \frac{2^4 \cdot 5 \cdot 13^4 \cdot 17^3}{3^{13}}$ & $A_4$ \\
[1.5ex] \hline & & &  \\ [-3ex]
$13$ & $\Q\left(\sqrt{13}\right)$ &  $\displaystyle -\frac{2^{12} \cdot 5^3 \cdot 11 \cdot 13^4}{3^{13}}$ & $A_4$ \\
[1.5ex] \hline & & &  \\ [-3ex]
$13$ & $\Q\left(\sqrt{13}\right)$ & $\displaystyle \frac{2^{18} \cdot 3^3 \cdot 13^4 \cdot 127^3 \cdot 139^3 \cdot 157^3 \cdot 283^3 \cdot 929}{5^{13} \cdot 61^{31}}$ & $A_4$  \\
[1.5ex] \hline & & &  \\ [-3ex]
$13$ & $\Q(\sqrt{13})$  & $\displaystyle \frac{2^{14} \cdot 5^2}{3^{13}}\left(2 \cdot 5 \cdot 251 \cdot 6373 \pm 13^2 \cdot 26251 \sqrt{13} \right)$ & $A_4$ \\
[1.5ex] \hline
\end{tabular}
\end{center}\bigskip

Unfortunately, we cannot confirm that there do not exist other counterexamples over $\Q$ in the case of $\ell = 13$. To do so, we would need to confirm that there are indeed only three noncuspidal rational points on $X_{S_4}(13)$. The genus of this curve is 3, so we know that it has only finitely many rational points, but its Jacobian has rank 3 (assuming the Birch and Swinnerton-Dyer conjecture), so the method of Chabauty-Coleman to bound its rational points does not necessarily apply. 

The following table lists the genera for the modular curves corresponding to the exceptional subgroups for $\ell \le 37$ which arise in Sections \ref{sec:Nsp} and \ref{sec:Nns}.\bigskip 

\begin{center}
\begin{tabular}{|c|ccccccccccc|}
\hline 
$H$ & $A_4$ & $A_4$ & $A_4$ & $S_4$ & $A_4$ & $S_4$ & $A_4$ &  $A_4$ & $A_5$ & $A_5$ & $S_4$  \\ \hline
$\ell$ & 5 & 7 & 11 & 11 & 13 & 13 & 17 & 19 & 29 & 31 & 37 \\ \hline
 $g(X_{H}(\ell))$ & 0 & 0 & 1 & 1 & 3 & 3  & 9 & 14 & 11 & 14 & 142
 \\ \hline
\end{tabular}
\end{center}
\bigskip

\begin{rmk} 
Genus formulas can be found in  \cite[Table 2.1]{cojocaru-hall2005} or \cite[II.2.1]{ligozat1976}. If $\ell \equiv \pm 3\mod{8}$, then $\PSL_2(\F_{\ell})$ does not contain a subgroup isomorphic to $S_4$, but it contains a subgroup isomorphic to $A_4$, and there is a lift of it in $\GL_2(\F_{\ell})$ whose image in $\PGL_2(\F_{\ell})$ is isomorphic to $S_4$. In this case we use the $A_4$ genus formula for $S_4$, as the associated modular curves are twists of each other.
\end{rmk}

The curve $X_{S_4}(11)$ has genus 1, and Ligozat showed that it is an elliptic curve with trivial Mordell-Weil group over $\Q$. This elliptic curve has Cremona label 121.a2. A rational point on this curve which is not cuspidal would give us a counterexample over $\Q$ to the local-global principle for the normalizer of a nonsplit Cartan. The one rational point is however a cusp, as Ligozat shows in \cite[II.4.4.1]{ligozat1976}, so there are no counterexamples over $\Q$ for $\ell = 11$. The Mordell-Weil group is also trivial if we base change to $K = \Q(\sqrt{-11})$, so there are no counterexamples over that field either. 





\bibliographystyle{alpha}
\bibliography{localglobal}

\end{document}